\documentclass[12pt, a4paper]{amsart}
\usepackage{amsmath,amssymb,amscd,amsfonts}

\usepackage{ifthen}
\usepackage[T1]{fontenc}
\usepackage[utf8]{inputenc}
\usepackage[all]{xy}
\usepackage{graphicx}
\usepackage{enumerate}
\usepackage{xspace}
\usepackage{pdfsync}
\usepackage{epic}
\usepackage{dsfont}

\newtheorem{theorem}{Theorem}[section]
\newtheorem{remark}[theorem]{Remark}

\newtheorem{lemma}[theorem]{Lemma}

\newtheorem{proposition}[theorem]{Proposition}

\newtheorem{corollary}[theorem]{Corollary}
\newtheorem{example}[theorem]{Example}

\begin{document}
\title[]{Albanese map of special manifolds: a correction}

\author{Fr\'ed\'eric Campana}
\address{Universit\'e Lorraine \\
Nancy \\ }

\email{frederic.campana@univ-lorraine.fr}



\date{\today}

\maketitle


\section{abstract} We show that any fibration of a `special' compact K\"ahler manifold $X$ onto an Abelian variety has no multiple fibre in codimension one. This statement strengthens and extends previous results of Kawamata and Viehweg when $\kappa(X)=0$. This also corrects the proof given in \cite{Ca04}, 5.3 which was incomplete. 

\section{Introduction}

The following statement is given in \cite{Ca04},Proposition 5.3:

\begin{theorem}\label{?} Let $a_X:X\to A_X$ is the Albanese map of $X$, assumed to be special. Then $a_X$ is onto, has connected fibres, and no multiple fibre in codimension one. 
\end{theorem}

Recall (\cite{Ca04}, definition 2.1 and Theorem 2.27) that $X$ is special if $\kappa(X,L)<p$ for any $p>0$, and any rank-one coherent subsheaf of $\Omega^p_X$. This implies that $X$ has no surjective meromorphic $g:X\dasharrow Z$ map onto a manifold $Z$ of general type and positive dimension $p$, and more generally exactly means that the `orbifold base' $(Z,D_g)$ of such a $g$, constructed out of its multiple fibres, is never of `general type'. See \S.2 below for some more details.

While the proofs given there for the first two properties (which generalise earlier results by Y. Kawamata and Kawamata-Viehweg in the case when $\kappa(X)=0$) are complete, the proof of the third property, even in the projective case, is not (as pointed out to me by K. Yamanoi and E. Rousseau\footnote{The problem comes from the potential multiplicity-one exceptional divisors of $a_X$ which are no longer $f$-exceptional if $g=q\circ f$, where $q:A\to B$ is a non-trivial torus quotient sending $D_{f}$ to an ample divisor of $B$. We overcome this difficulty in the second step of the proof, by cutting the fibration by means of Poincar\'e reducibility. The first step is the same as in \cite{Ca04}.}). The aim of this note is to correct it in the projective case (but only partially in the compact K\"ahler case) by using the main result of \cite{CC}, itself based on \cite{BC} (or, alternatively, \cite{CP17}). 

\begin{theorem}\label{!} Let $f:X\to A$ be a holomorphic map from the connected, compact, K\"ahler manifold to a compact complex torus $A$. If $X$ is special, then:

1.  $f$ is surjective, 

2. If $f$ is the Albanese map of $X$, then $f$ has connected fibres, and:

3. $f$ has no multiple fibre\footnote{We use here the `inf' multiplicities, not the more classical `gcd' version, which would give a weaker result.} in codimension one if the fibres of $f$ are connected, and if $A$ is an Abelian variety (equivalently: $D_{f}=0$ if $(A,D_f)$ is the `orbifold base' of $f$). 
\end{theorem}

\begin{remark} 1. Although the assumption that $A$ is an Abelian variety in Claim 3 is certainly not necessary, our proof uses \cite{BC} (or alternatively \cite{CP17}), presently known only for $A$ projective. Even with a K\"ahler version of \cite{BC} or  \cite{CP17}, we could not treat the K\"ahler case in general with the prsent arguments, because we use Poincar\'e reducibility. 

2. An easy case of Claim 3, not covered by Theorem \ref{!} is when the algebraic dimension $a(X)$ of $X$ vanishes (and more generally, if $a(A_X)=0$, $A_X$ being the Albanese torus of $X$), because then $A_X$ does not contain any effective divisor. The argument of the second step of the proof below shows that the conclusion of Theorem \ref{!}.(3) still holds true if $A=T\times V$, where $T$ is a compact torus of algebraic dimension zero, and $V$ an abelian variety. 

3. Claims 1 and 2, and their proofs, still apply when $X$ is `weakly special', which means that no finite \'etale cover $X'$ of $X$ has a surjective meromorphic fibration onto a positive-dimensional projective manifold of general type. Weak-specialness coincides with specialness for curves and surfaces, but differ from dimension $3$ on, by examples of Bogomolov-Tschinkel. 
\end{remark}

\section{Multiple fibres of maps to complex tori.}

If $f:X\to Y$ is a fibration, that is: a holomorphic map with connected fibres onto a (smooth) compact complex manifold $Y$, we define its `orbifold base' $D_f$ as follows: for every prime divisor $E\subset Y$, let $f^*(E):=\sum_{k\in K}t_k.D_k+R$, where the $D_k's$ are the pairwise distinct prime divisors of $X$ surjectively mapped by $f$ onto $E$, while $R$ is an effective $f$-exceptional divisor of $X$. The multiplicity\footnote{We use here the `inf' multiplicity instead of the usual `gcd' multiplicity for reasons explained in \cite{Ca04}.} $m_f(E)$ of the generic fibre of $f$ over $E$ is then defined as: $m_f(E):=inf\{t_k,k\in K\}$, and then $D_f:=\sum_{E\subset Y}(1-\frac{1}{m_f(E)}).E$, which is an effective $\Bbb Q$-divisor of $Y$ (note indeed that this sum is finite since $m_f(E)=1$ for all, but finitely many of the $E's$).

\medskip

\begin{proof} (of Theorem \ref{!}) 

Claim 1. Assume by contradiction that $f(X):=Z\neq  A$. After Ueno's theorem (\cite{U}, p. 120), there is a quotient torus $q:A\to B$ such that $q(Z)\subset B$ is of dimension $p>0$, and of general type. The composed map $g:=q\circ f:X\to Z$ thus contradicts the specialness of $X$, since $\kappa (X,g^*(K_Z))=p$ and $g^*(K_Z)\subset \Omega^p_X$. 

Claim 2. Let $Z$ be a smooth model of the Stein factorisation $Z_0$ of $f$. We replace $X$ by $Z$, which is still special since dominated by $X$. We may thus assume that $f$ is generically finite. By \cite{kaw}, Theorem 23, replacing $Z$ by a suitable finite \'etale cover which is still special by \cite{Ca04}, \S.5.5, we may assume that $Z$ fibres over over a manifold of general type if $Z$ is not birational to $A$.  This again contradicts the specialness of $Z$ and thus of $X$.

Claim 3. Assume that $D_{f}\neq 0$. We shall first treat the case when $D_f$ is ample, and reduce to this case in a second step. 

\medskip

$\bullet$ Let thus $D_f$ be ample on $A$.  By a flattening of $f$, followed by suitable blow-ups of $X$ and $A$, we may assume (see \cite{Ca04},Lemma 1.3) that $f=v\circ f'$, where $v:A'\to A$ is bimeromorphic with $A'$ smooth, and $f':X\to A'$ is a fibration such that its orbifold base $D':=D_{f'}=\overline{D}+E'$, with $E'$ effective and $v$-exceptional, $\overline{D}$ is the strict transform of $D_f$ in $A'$. This shows that in particular $f^*(K_{A'}+\overline{D})\subset \Omega^p_X, p:=dim(A)$. The following lemma \ref{lem} shows that the line bundle $K_{A'}+\overline{D}$ has Kodaira dimension $dim(A)=p$, since $D_f$ is ample on $A$. This contradicts the specialness of $X$ if $D_f$ is assumed to be ample.

\begin{lemma}\label{lem}(\cite{Ca04},1.14) Let $v:Y'\to Y$ be a bimeromorphic map between compact connected complex manifolds. Let $D$ be an effective $\Bbb Q$-divisor on $Y$, and let $\overline{D}$ be its strict transform on $Y'$ . Then:

1. $K_{Y'}+\overline{D}=v^*(K_Y+\varepsilon.D)+E', \forall \varepsilon>0$ small enough, $E'$ effective.

2. If $\kappa(Y)\geq 0$, then $\kappa(Y',K_{Y'}+\overline D)=\kappa(Y,K_Y+D)$.
\end{lemma} 

\begin{proof} Claim 1. Let the rational numbers $a_i,b_i$ be defined as follows: $v^*(D)=\overline{D}+\sum_i b_i.E_i,K_{Y'}=v^*(K_Y)+\sum_i a_i.E_i$, where the $E_i's$ are the exceptional divisors of $v$. Then:  $a_i>0, b_i\geq 0,\forall i$. Then:

 $K_{Y'}+\overline{D}=v^*(K_{Y})+\sum a_i.E_i+(1-\varepsilon).\overline{D}+v^*(\varepsilon.D)-(\varepsilon.\sum b_i.E_i)=v^*(K_{Y}+\varepsilon.D)+(1-\varepsilon).\overline{D}+\sum(a_i-\varepsilon.b_i).E_i$. We thus get the first claim when $0<\varepsilon\leq min\{1, \frac{a_i}{b_i}, \forall i\}$, with $E':=(1-\varepsilon).\overline{D}+\sum(a_i-\varepsilon.b_i).E_i$.
 
 Claim 2. Let $\varepsilon>0$ be as above. Since $K_Y$ is assumed to be $\Bbb Q$-effective, so is $(1-\varepsilon).v^*(K_Y)$, and so $K_{Y'}+\overline{D}=\varepsilon. v^*(K_Y+D)+E"$, with:
 $E':=E'+(1-\varepsilon).v^*(K_Y)$, which is $\Bbb Q$-effective.\end{proof}

\begin{remark} The main property used in Lemma \ref{lem} is that $a_i >0, \forall i$, i.e: the lift of the canonical sheaf under a modification of a smooth manifold vanishes on the exceptional divisor. This property does not hold true for sheaves of forms of degree less than the dimension. We thus need to use a less direct route, which requires deeper ingredients. 
\end{remark} 

We shall now reduce to the case when $D_f$ is ample. 

\medskip

$\bullet$ Let $L:=\mathcal{O}_{A}(m.D_{f})$ be the line bundle on $A$ with a section vanishing on some integral multiple $m.D_{f}$ of the orbifold base of $f$. Let $T\subset Aut^0(A)$ be the connected component of the group of translations of $A$ preserving $L$. Let $q:A\to B:=A/T$ be the quotient map. Then $D_{f}=\frac{1}{m}.q^*(D)$ for some {\it ample} effective divisor $D$ on $B$ (see, for example \cite{D}, Th\'eor\`eme 5.1). Since $D_{f}\neq 0$, $p:=dim(B)>0$.

We use Poincar\'e reducibility to reduce to the case that $A=B\times B'$, where $B'\to T$ is a suitable finite \'etale cover of $T$. This indeed amounts to replace $X$ and $D_f$ by the corresponding  finite \'etale covers, preserving all of our hypothesis, in particular the specialness of $X$ (by \cite{Ca04}, \S.5.5, again). 

We now consider the composed fibrations $g:=q\circ f:X\to B$, and $h:=q'\circ f:X\to B'$, where $q': A\to B'$ is the second projection. By \cite{CC}, 2.4, the general fibre $X_{b'}:=h^{-1}(b'), b'\in B'$ of $g'$, is special\footnote{The statement of \cite{CC} is only given for $X$ projective, and $f$ the Albanese map of $X$, but it is easy to check that the proof, derived from \cite{BC}, or \cite{CP17}, through \cite{CC}, 2.6, applies for $X$ compact and any map to an Abelian variety.}. Let $f_{b'}:X_{b'}\to A_{b'}:=B\times \{b'\}$ be the restriction of $f$ to $X_{b'}$. The orbifold base $(A_{b'},D_{f_{b'}})$ of $f_{b'}$ is then nothing, but $(A,D_f)\cap A_{b'}$, and is thus ample on $A_{b'}\cong B$. By the first part of the proof, this contradicts the specialness of $X_{b'}$ if $D_{f_{b'}}\neq 0$. Thus $D_{f}\cap A_{b'}=0$, and so $D_f=0$.\end{proof}

\end{document}